\documentclass[reqno]{amsart}

\usepackage{soul}

\usepackage{amsmath,times}
\usepackage{amsthm,amssymb}
\usepackage{latexsym}
\usepackage{stackrel}
\usepackage{graphicx}
\usepackage{color}
\usepackage{enumerate}
\usepackage{float}
\usepackage{array}
\newcolumntype{L}{>{$}l<{$}} 

\usepackage{comment}
\usepackage{cite}

\usepackage{url}
\usepackage{breakurl}

\theoremstyle{definition}

\newtheorem{theorem}{Theorem}[section]
\newtheorem*{theorem*}{Theorem}
\newtheorem{lemma}[theorem]{Lemma}

\newtheorem{proposition}[theorem]{Proposition}

\theoremstyle{definition}
\newtheorem{example}[theorem]{Example}
\newtheorem{definition}[theorem]{Definition}
\newtheorem{corollary}[theorem]{Corollary}
\newtheorem*{corollary*}{Corollary}
\newtheorem{remark}[theorem]{Remark}
\newtheorem{question}[theorem]{Question}
\newtheorem{remarks}[theorem]{Remarks}

\newenvironment{sketchproof}{%
  \proof}{\endproof}

\numberwithin{equation}{section}

\newcommand{\HR}{\operatorname{HR}}
\newcommand{\wHR}{\operatorname{HR_w}}

\newcommand{\sub}{\operatorname{sub}}
\newcommand{\Hom}{\operatorname{Hom}}

\newcommand{\rank}{\operatorname{rk}}

\title{On Hodge-Riemann Cohomology Classes}

\author{Julius Ross and Matei Toma}

\address{Department of Mathematics, Statistics, and Computer Science, University of Illinois at Chicago, 322 Science and Engineering Offices (M/C 249), 851 S. Morgan Street, Chicago, IL 60607
 }
\email{juliusro@uic.edu}

\address{Universit\'e de Lorraine, CNRS, IECL, F-54000 Nancy, France
 }
\email{matei.toma@univ-lorraine.fr}

\date{\today}

\begin{document}
\bibliographystyle{plain}

\begin{abstract}
We prove that Schur classes of nef vector bundles are limits of classes that have a property analogous to the Hodge-Riemann bilinear relations.   We give a number of applications, including (1)  new log-concavity statements about characteristic classes of nef vector bundles (2) log-concavity statements about Schur and related polynomials (3) another proof that normalized Schur polynomials are Lorentzian.
\end{abstract}
\maketitle

\section{Introduction}
Since the dawn of time, human beings have asked some fundamental questions: who are we?  why are we here?  is there life after death?  Unable to answer any of these,  in this paper we will consider cohomology classes on a compact projective manifold that have a property analogous to the  Hard-Lefschetz Theorem and Hodge-Riemann bilinear relations.


To state our results let $X$ be a projective manifold of dimension $d\ge 2$.   We say that a cohomology class $\Omega\in H^{d-2,d-2}(X;\mathbb R)$ has the \emph{Hodge-Riemann property} if the intersection form
$$ Q_{\Omega}(\alpha,\alpha') := \int_X \alpha \Omega \alpha' \text{ for } \alpha,\alpha'\in H^{1,1}(X;\mathbb R)$$
has signature $(+,-,-,\ldots,-)$.    We write $$\HR(X)= \{ \Omega \text{ with the Hodge Riemann property}\}$$ 
and $\overline{\HR}(X)$ for its closure.\medskip

This definition is made in light of the fact that the classical Hodge-Riemann bilinear relations say precisely that if $L$ is an ample line bundle on $X$, then $c_1(L)^{d-2}$ is in $\HR(X)$.   A natural question, initiated by Gromov \cite{Gromov90}, is if there are other cohomology classes that have this property, and our first result answers this in terms of certain characteristic classes of vector bundles.

\begin{theorem*}[$\subseteq$ Theorem \ref{thm:derivedschurclassesareinHR}] Let $E$ be a nef vector bundle on $X$ and $\lambda$ be a partition of $d-2$.  Then the Schur class
$s_{\lambda}(E)$ lies in $\overline{\HR}(X)$.
\end{theorem*}

 In fact we can do better; for each $i$ define the \emph{derived Schur polynomials} $s_{\lambda}^{(i)}$ by requiring that
$$ s_{\lambda}(x_1+t,\ldots,x_e+t) = \sum_{i=0}^{|\lambda|} s_{\lambda}^{(i)}(x_1,\ldots,x_e) t^i.$$
\begin{theorem*}[$\subseteq$ Theorem \ref{thm:derivedschurclassesareinHR}]  Let $E$ be a nef vector bundle on $X$ and $\lambda$ be a partition of $d-2+i$.  Then the derived Schur class
$s_{\lambda}^{(i)}(E)$ lies in $\overline{\HR}(X)$.
\end{theorem*}

 \clearpage
We prove moreover:
\begin{itemize}
\item Analogous statements hold for monomials of derived Schur classes of possibly different nef vector bundles (Theorem \ref{thm:monomialsderivedschurclassesareinHR}).
\item If $E$ is perturbed by adding a sufficiently small ample class, then  $s_{\lambda}(E)$ lies in $\HR(X)$ (rather than in just the closure) (Remark \ref{rmk:comparisonwitholdpaper}).
\item The above holds even in the setting of compact K\"ahler manifolds, where nefness of $E$ is taken in the metric sense following Demailly-Peternell-Schneider (Theorem \ref{thm:derivedschurclassesareinHR:Kahler}).
\end{itemize}

\begin{center}
*
\end{center}

Our above result is interesting even in the case that $E=\oplus_{i=1}^e L_i$ is a direct sum of ample line bundles, from which we deduce that the Schur polynomial $s_{\lambda}(c_1(L_1),\ldots,c_1(L_e))$ lies in $\overline{\HR}(X)$.   As a concrete example, $s_{(1,1)}(x_1,x_2) = x_1^2 + x_1 x_2 + x_2^2$, so if $L_1$ and $L_2$ are ample line bundles on a fourfold the class
\begin{equation}c_1(L_1)^2 + c_1(L_1)c_1(L_2) + c_1(L_2)^2 \in \overline{\HR}(X).\label{eq:simplexample}\end{equation}
As already noted, the classical Hodge-Riemann bilinear relations tell us that the classes $c_1(L_1)^2$ and $c_1(L_2)^2$ both lie in $\HR(X)$, and it was proved by Gromov \cite{Gromov90} that the mixed term $c_1(L_1)c_1(L_2)$ also lies in $\HR(X)$.  However in general having the Hodge-Riemann property is not preserved under taking convex combinations, and thus \eqref{eq:simplexample} is new.   \\

From these considerations it is natural to ask which universal combinations of characteristic classes of ample (resp.\ nef) vector bundles lie in $\HR(X)$ (resp. $\overline{\HR}(X)$).   Although we do not know the full answer to this, the following is a contribution in this direction.

\begin{theorem*}[$\subseteq$ Theorem \ref{thm:poyla}]
Let $E$ be a nef vector bundle on a projective manifold of dimension $d$, and $\lambda$ be a partition of $d-2$.  Suppose $\mu_0,\ldots,\mu_{d-2}$ is a P\'olya frequency sequence of non-negative real numbers.  Then the combination
$$ \sum_{i=0}^{d-2} \mu_i s_{\lambda}^{(i)}(E) c_1(E)^{i}$$
lies in $\overline{\HR}(X)$.
\end{theorem*}

\begin{center}
*
\end{center}

As an application of these results we are able to give various new inequalities between characteristic classes of nef vector bundles.   Continuing to assume $X$ is projective of dimension $d$, let $\lambda$ and $\mu$ be partitions of length $|\lambda|$ and $|\mu|$ respectively and assume $|\lambda| + |\mu|\ge d$.

\begin{theorem*}[= Theorem \ref{thm:generalizedKT}]
Assume $E,F$ are nef vector bundles on $X$.     Then the sequence
\begin{equation}\label{eq:generalizedKT}i\mapsto \int_X  s_{\lambda}^{(|\lambda|+|\mu|-d-i)}(E)s_{\mu}^{(i)}(F) \end{equation}
is log-concave
\end{theorem*}

As a particular case, we get that if $E$ is a nef vector bundle and $\lambda$ a partition of $d$, then 
$$j\mapsto \int_X s_{\lambda}^{(j)}(E)c_1(E)^{j} $$
is log-concave, which as a special case says the map
$$i\mapsto \int_X c_i(E)c_1(E)^{d-i} $$
is also log-concave.  One should think of these statements as higher-rank analogs of the Khovanskii-Tessier inequalities.  We even get combinatorial applications of this, such as the following:

\begin{corollary*}[= Corollary \ref{cor:combinatoric}]
Let $\lambda$ and $\mu$ be partitions, and let $d$ be an integer with $d\le |\lambda| + |\mu|$.    Assume  $x_1,\ldots,x_e,y_1,\ldots,y_f\in \mathbb R_{\ge 0}$.   Then the sequence
$$i\mapsto s^{(|\lambda| + |\mu| -d +i)}_{\lambda}(x_1,\ldots,x_e) s_{\mu}^{(i)}(y_1,\ldots,y_f)$$
is log concave.
\end{corollary*}

\begin{corollary*}[= Corollary \ref{cor:combinatoric2}]
Let $\lambda$ be a partition and $x_1,\ldots,x_e\in \mathbb R_{\ge 0}$.  Then the sequence
$$i\mapsto s_{\lambda}^{(i)}(x_1,\ldots,x_e)$$
is log-concave.
\end{corollary*}

This last statement has been known for a long time for the partition $\lambda=(e)$, for then the derived Schur polynomials become the elementary symmetric polynomials $c_i$ (see Example \ref{ex:Chernclasses}).  Then more is true namely, $i\mapsto c_i(x_1,\ldots,x_e)$ is ultra-log concave - a result which is due to Newton \cite{Newton} (see, for example, \cite[Chap.\ 11]{Cvetkovski} for a modern treatment).

\medskip

As a final application we show how knowing that Schur classes of nef bundles lie in $\overline{\HR}(X)$ gives another proof of a result of Huh-Matherne-M\'esz\'aros-Dizier \cite{Huhetal} that the normalized Schur polynomials are Lorentzian.

\subsection{Comparison with  previous work: } There is some overlap between Theorem \ref{thm:derivedschurclassesareinHR} and our original work on the subject in \cite{RossTomaHR}.  A principal difference is that in \cite{RossTomaHR} we show that derived Schur classes of ample bundles have the Hodge-Riemann property,  whereas here we settle in merely showing these classes are limits of classes with this property.   So even though logically many of our results follow from \cite{RossTomaHR}, the proofs we give here are simpler and substantially shorter.  In fact, our account here does not depend on any of the parts of  \cite{RossTomaHR}  and is self-contained relying only on a few standard techniques in the field (as contained say in \cite{Lazbook2}).   The main tools we use are the Bloch-Gieseker theorem, and the cone classes of Fulton-Lazarsfeld that express Schur classes as pushforwards of certain Chern classes (which builds on the determinantal formula of Kempf-Laksov \cite{Kempf}) .   The material on the non-projective case in \S\ref{sec:Kahler}, on convex combinations in \S\ref{sec:combinations} and on inequalities in \S\ref{sec:inequalities} is all new. 

We refer the reader to \cite{RossTomaHR} for a survey of other works concerning Hodge-Riemann classes.   Although there are many places in which log-convexity and Schur polynomials meet (e.g. \cite{Chen,Okounkov-logconcave,Huhetal,Lam,Gao,Richards}) we are not aware of any previous inequalities that cover precisely those studied here.

\subsection{Organization of the paper: }  \S\ref{sec:notationandconvention}, \S\ref{sec:derivedschurclasses} and \S\ref{sec:coneclasses} contain preliminary material on Schur polynomials, derived Schur polynomials and cone classes. We also include in  \S\ref{sec:fulton-lazarsfeld} a self-contained proof of a theorem of Fulton-Lazarsfeld concerning positivity of (derived) Schur polynomials.    The main theorems about derived Schur classes having the Hodge-Riemann property is proved in  \S\ref{sec:schurclassesareinHR}, and in  \S\ref{sec:Kahler} we explain how this extends to the non-projective case.    In  \S\ref{sec:combinations} we consider convex combinations of Hodge-Riemann classes, and in  \S\ref{sec:inequalities} we give our application to inequalities and our proof that normalized Schur polynomials are Lorentzian.

\subsection{Acknowledgments: } The first author thanks Ivan Cheltsov and Jinhun Park for their hospitality in Pohang and for the opportunity to present this work at the Moscow-Shanghai-Pohang Birational Geometry, K\"ahler-Einstein Metrics and Degenerations conference.   The first author is supported by NSF grants DMS-1707661 and DMS1749447.

\section{Notation and convention}\label{sec:notationandconvention}

We work throughout over the complex numbers.   For the majority of the paper we will take $X$ to be a projective manifold (which we always assume is connected), and $E$ a vector bundle (which we always assume to be algebraic).  Given such a vector bundle $E$ we denote by $\pi: \mathbb P(E)\to X$ the space of one-dimensional quotients of $E$, and by $\pi:\mathbb P_{sub}(E)\to X$ the space of one-dimensional subspaces of $E$.   We say that a vector bundle $E$ is ample (resp.\ nef) if the hyperplane bundle $\mathcal O_{\mathbb P(E)}(1)$ on $\mathbb P(E)$ is ample (resp.\ nef).  

We will make use of the formalism of $\mathbb Q$-twisted bundles 
(see \cite[Section 6.2, 8.1.A]{Lazbook2}, \cite[p457]{Miyoka}).  Given a vector bundle $E$ on $X$ of rank $e$ and an element $\delta \in N^1(X)_{\mathbb Q}$ the $\mathbb Q$-twisted bundle denoted $E\langle \delta \rangle$ is a formal object understood to have Chern classes defined by the rule
\begin{equation}c_p(E\langle \delta \rangle) := \sum_{k=0}^p\binom{e-k}{p-k} c_k(E) \delta^{p-k} \text{ for } 0\le p\le e.\label{eq:defcherntwisted}\end{equation}
Here and henceforth we abuse notation  and write $\delta$ also for its image under $N^1(X)_{\mathbb Q} \to H^2(X;\mathbb Q)$, so the above intersection is taking place in the cohomology ring $H^*(X)$.

By the rank of $E\langle \delta \rangle$ we mean the rank of $E$.   The above definition is made so if $\delta= c_1(L)$ for a line bundle $L$ on $X$ then $$c_p(E\langle c_1(L)\rangle) = c_p(E\otimes L).$$   If $E$ has Chern roots given by $x_1,\ldots,x_e$ then $E\langle \delta\rangle$ is understood to have Chern roots $x_1+\delta,\ldots,x_e+\delta$.    The twist of an $\mathbb Q$-twisted bundle is given by the rule $E\langle \delta \rangle \langle \delta' \rangle = E\langle \delta + \delta' \rangle$.   That \eqref{eq:defcherntwisted} continues to hold when $E$ is an $\mathbb Q$-twisted bundle is an elementary calculation  - for convenience of the reader, we omit the proof.
  
We say that $E\langle \delta \rangle$  is ample (resp.\ nef) if the class $c_1(\mathcal O_{\mathbb P(E)}(1)) + \pi^* \delta$ is ample (resp.\ nef) on $\mathbb P(E)$.  

Suppose $p(x_1,\ldots,x_e)$ is a homogeneous symmetric polynomial of degree $d'$ and $E$ is a $\mathbb Q$-twisted vector bundle of rank $E$ on $X$ with Chern roots $\tau_1,\ldots,\tau_e$.  Then we have the well-defined characteristic class $$p(E): = p(\tau_1,\ldots,\tau_e) \in H^{d',d'}(X;\mathbb R).$$

By abuse of notation we let $c_i$ denote the $i$th elementary symmetric polynomial, so $c_i(E)\in H^{i,i}(X;\mathbb R)$ is unambiguously defined as the $i$th-Chern class of $E$.

\section{Derived Schur Classes}\label{sec:derivedschurclasses}
By a partition $\lambda$ of an integer $b\ge 1$ we mean a sequence $0\le \lambda_N\le \cdots \le \lambda_1$ such that $|\lambda|:=\sum_{i} \lambda_i =b$.       For such a partition, the Schur polynomial $s_{\lambda}$ is the symmetric polynomial of degree $|\lambda|$ in $e\ge 1$ variables given by
$$s_{\lambda} = \det \left(\begin{array}{ccccc} c_{\lambda_1} & c_{\lambda_1+1} &\cdots &c_{\lambda_1 +N-1}\\
c_{\lambda_2-1} & c_{\lambda_2} &\cdots &c_{\lambda_2 +N-2}\\
\vdots & \vdots & \vdots & \vdots\\
c_{\lambda_N-N+1} & c_{\lambda_N -N +2}&\cdots &c_{\lambda_N}\\
\end{array}\right)$$
where $c_i$ denotes the $i$-th elementary symmetric polynomial.    

 We will have use for the following symmetric polynomials associated to Schur polynomials.

\begin{definition}\label{def:derivedschur}[Derived Schur polynomials]
Let $\lambda$ be a partition.  For any $e\ge 1$  we define $s_{\lambda}^{(i)}(x_1,\ldots,x_e)$ for $i=0,\ldots,|\lambda|$ by requiring that
$$ s_{\lambda}(x_1+t,\ldots,x_e+t) = \sum_{i=0}^{|\lambda|} s_{\lambda}^{(i)}(x_1,\ldots,x_e) t^i \text{ for all } t\in \mathbb R.$$

\end{definition}

In fact $s_{\lambda}^{(i)}$ depends also on $e$ but we drop that from the notation.   By convention we set  $s_{\lambda}^{(i)}=0$ for $i\notin \{0,\ldots,|\lambda|\}$.   For $0\le i\le |\lambda|$, clearly  $s_{\lambda}^{(i)}$ is a homogeneous symmetric polynomial of degree $|\lambda|-i$ and  $s_{\lambda}^{(0)} = s_{\lambda}$.   

Thus for any $\mathbb Q$-twisted vector bundle $E$ of rank $e$ we have classes
$$s_{\lambda}^{(i)}(E)\in H^{|\lambda|-i,|\lambda|-i}(X;\mathbb R),$$ 
and by construction if $\delta \in N^1(X)_{\mathbb Q}$ then
$$ s_{\lambda}(E\langle \delta\rangle) = \sum_{i=0}^{|\lambda|} s_{\lambda}^{(i)}(E) \delta^i.$$

\begin{example}[Chern classes]\label{ex:Chernclasses}
Consider the  partition of $\lambda = (p)$ consisting of just one integer.    Then $s_{\lambda}=c_p$, and from standard properties of Chern classes of a tensor product if $\rank E=e\ge p$ then
$$s_{\lambda}^{(i)}(E) = \binom{e-p+i}{i} c_{p-i}(E)\text{ for all } 0\le i\le p.$$
\end{example}

\begin{example}[Derived Schur polynomials of Low degree]
We list some of the derived Schur classes of low degree for a bundle $E$ of rank $e$.  First 
$$s_{(1)} = c_1, \quad s_{(1)}^{(1)} = e$$
and for $e\ge 2$,
\begin{align*}
s_{(2,0)} &= c_2 & s_{(2,0)}^{(1)} &= (e-1) c_1 & s_{(2,0)}^{(2)}&=\binom{e}{2}\\
s_{(1,1)} &= c_1^2 -c_2,& s_{(1,1)}^{(1)}&=(e+1)c_1 & s_{(1,1)}^{(2)}&= \binom{e+1}{2}
\end{align*}
and for $e\ge 3$,
\begin{align*}
s_{(3,0,0)} &= c_3 & s_{(3,0,0)}^{(1)} &= (e-2)c_2 & s_{(3,0,0)}^{(2)}&= \binom{e-1}{2} c_1\\
&&&& s_{(3,0,0)}^{(3)}&=\binom{e}{3} \\
s_{(2,1,0)}&= c_1c_2 -c_3&  s_{(2,1,0)}^{(1)}&= 2c_2 + (e-1)c_1^2 & s_{(2,1,0)}^{(2)}&=(e^2-1) c_1 \\
&&&&  s_{(2,1,0)}^{(3)}&=2\binom{e+1}{3}\\
s_{(1,1,1)}&= c_1^3 -2c_1c_2 + c_3 & s_{(1,1,1)}^{(1)}&=(e+2) (c_1^2-c_2) &s_{(1,1,1)}^{(2)}&= \binom{e+2}{2} c_1 \\
&&&&s_{(1,1,1)}^{(3)}&= \binom{e+2}{3}
\end{align*}
\end{example}

\begin{example}[Lowest Degree Derived Schur Classes]\label{ex:lowestdegree}
Suppose $e\ge \lambda_1$.  Then we can write the Schur polynomial as a sum of monomials 
$$s_{\lambda} (x_1,\ldots,x_e)= \sum_{|\alpha| = |\lambda|} c_\alpha x_1^{\alpha_1} \cdots x_e^{\alpha_e}$$
where  $c_{\alpha}\ge 0$ for all $\alpha$ (in fact the $c_{\alpha}$ count the number of semistandard Young tableaux of weight $\alpha$ whose shape is conjugate to $\lambda$).  Since $e\ge \lambda_1$, $s_{\lambda}$ is not identically zero, so at least one of the $c_{\alpha}$ is strictly positive.  Thus in the expansion
$$s_{\lambda}(x_1+t,\ldots, x_e+t)=\sum_{i=0}^{|\lambda|} s_{\lambda}^{(i)}(x_1,\ldots,x_e) t^i $$
the coefficient in front of $t^{|\lambda|}$ is strictly positive, i.e.  $s_{\lambda}^{(|\lambda|)}>0$.   

So, in terms of characteristic classes, if $E$ has rank at least $\lambda_1$  then $$s_{\lambda}^{(|\lambda|)}(E) \in H^{0}(X;\mathbb R)=\mathbb R$$ is strictly positive. 
\end{example}

\section{Cone Classes}\label{sec:coneclasses}

We will rely on a construction exploited by Fulton-Lazarsfeld that express Schur classes as the pushforward of Chern classes, and we include a brief description here.    Let $E$ be a vector bundle of rank $e$ on $X$ of dimension $d$ and suppose $0\le \lambda_N\le \lambda_{N-1} \le \cdots \le \lambda_1$ is a partition of length $|\lambda| = b\ge 1$ and $\lambda_1\ge e$. 
  Set $a_i: = e + i -\lambda_i$ and fix a vector space $V$ of dimension $e+N$.  Then it is possible to find a nested sequence of subspaces $0\subsetneq A_1\subsetneq A_{2} \subsetneq\cdots \subsetneq A_N\subset V$ with $\dim(A_i) = a_i$.

We set $F: = V^*\otimes E=\Hom(V,E)$ and let $f + 1 = \rank(F) = e(e+N)$.   Then inside $F$ define
$$\hat{C} : = \{ \sigma\in \Hom(V,E): \dim \ker(\sigma(x))\cap A_i \ge i \text{ for all } i=1,\ldots, N \text{ and } x\in X\}$$
which is a cone in $F$.
Finally set $$C = [\hat{C}] \subset \mathbb P_{\sub}(F).$$

\begin{proposition}\label{prop:pushforwardschur}
$C$ has codimension $b$ and dimension $d+f-b$, has irreducible fibers over $X$ and is flat over $X$ (in fact it is locally a product).  Moreover if 
\begin{equation}0 \to \mathcal O_{\mathbb P_{\sub}(F)}(-1) \to \pi^* F \to U \to 0\label{eq:tautologicalsequence}\end{equation} 
is the tautological sequence then
\begin{equation}s_{\lambda}(E) = \pi_* c_{f}(U|_C).\label{eq:pushforwasschur}\end{equation}
\end{proposition}
\begin{proof} This is described by Fulton-Lazarsfeld in \cite{FultonLazarsfeld}.   An account (that is written for the the case $|\lambda| = d$) can be found in \cite[(8.12)]{Lazbook2} and an account for general $|\lambda|$ is given in \cite[Proposition 5.1]{RossTomaHR} that is based on \cite{FultonIT}.   We remark that in \cite[Proposition 5.1]{RossTomaHR} we made the additional assumption that $N\ge b$ and $e\ge 2$, but have since realized these are not necessary (we used this to ensure that $f\ge b$, but this actually follows immediately from $e\ge \lambda_1$).
\end{proof}

This extends to $\mathbb Q$-twisted bundles $E' = E\langle \delta\rangle$.  Here we identify $$P':=\mathbb P_{\sub}(F\langle \delta \rangle)\stackrel{\pi}{\to} X$$ with $\mathbb P_{\sub}(F)\stackrel{\pi}{\to} X$ but the quotient bundle $U$ on $P'$ is replaced by $U': =U \langle \pi^* \delta\rangle$.  We consider the same cone $[C]\subset P'$.       Then \eqref{eq:pushforwasschur} still holds in the sense that
\begin{equation}s_{\lambda}(E') = \pi_* c_{f}(U'|_C).\label{eq:pushforwasschur:twisted}\end{equation}
To see this, observe that as $\delta\in N^1(X)_{\mathbb Q}$ we have $\delta =\frac{1}{m} c_1(L)$ for some $m\in \mathbb Z$ and line bundle $L$.  Then for $t$ divisible by $m$ 
\begin{equation}\pi_* c_{f}(U\langle t\pi^* \delta\rangle)  = \pi_* c_{f}(U \otimes \pi^* L^{\frac{t}{m}}) =  s_{\lambda} (E\otimes L^{\frac{t}{m}}) = s_{\lambda}(E\langle t\delta\rangle)\label{eq:pushforwasschurt}\end{equation}
where the second equality uses \eqref{eq:pushforwasschur}.  But both sides of \eqref{eq:pushforwasschurt} are polynomials in $t$, so since this equality holds for infinitely many $t$ it must hold for all $t\in \mathbb Q$, in particular when $t=1$ which gives \eqref{eq:pushforwasschur:twisted}.


A key feature we will rely on is that if $E'$ is assumed to be nef then so is $U'$.  For if $E'$ is nef then so is $F':=F\langle \delta\rangle$ and the formal surjection $F'\to U'$ coming from \eqref{eq:tautologicalsequence} implies that $U'$ is also nef (see \cite[Lemma 6.2.8]{Lazbook2} for these properties of nef $\mathbb Q$-twisted bundles).\medskip

Another extension is to the product of Schur classes of possibly different vector bundles $E_1,\ldots,E_p$ on $X$.    Let $\lambda^{1},\ldots,\lambda^{p}$ be partitions and assume $\rank(E_j)\ge\lambda_1^j$ for  $j=1,\ldots,p$.  We consider again the corresponding cones $C_i$ that sit inside  $F_i:=\Hom(V_i,E_i)$ for some vector space $V_i$.   We may consider the fiber product $C:=C_1\times_X C_2 \times_X \cdots \times_X C_p$ inside $\oplus_j \Hom(V_i,E_i)=:F$ and its projectivization $[C]\subset \mathbb P_{\sub}(F)$.  Then, using that each $C_i$ is flat over $X$,   if $U$ is the tautological vector bundle on $ \mathbb P_{\sub}(F)$ of rank $f$ we have
\begin{equation}\pi_* c_{f}(U|_C) = \prod_j s_{\lambda^{j}}(E_j)\label{eq:productschur}\end{equation}
(see \cite[8.1.19]{Lazbook2}, \cite[Sec 3c]{FultonLazarsfeld}).

\section{Fulton-Lazarsfeld Positivity}\label{sec:fulton-lazarsfeld} 
Using the cone construction we quickly get the following positivity statement, which is essentially a weak version of a result of Fulton-Lazarsfeld \cite{FultonLazarsfeld}.  For the reader's convenience we include the short proof here.

\begin{proposition}\label{prop:fultonlazderivedshur}Let $X$ be smooth and projective of dimension $d$, $\lambda$ be a partition of length $d+i$ for some $i\ge 0$ and $E$ be an $\mathbb Q$-twisted nef vector bundle.   Then $\int_X s_{\lambda}^{(i)}(E)\ge 0$. \end{proposition}

\begin{proof}

We first claim that if $E$ is a nef $\mathbb Q$-twisted bundle of rank $d$ on an irreducible projective variety $X$ of dimension $d$ then $\int_X c_d(E)\ge 0$.   By taking a resolution we may assume $X$ is smooth.   Let $h$ be an ample class on $X$.  By the Bloch-Gieseker Theorem \cite{BlochGieseker} we have $\int_X c_{d}(E\langle t h\rangle) \neq 0$ for all $t>0$ since $E\langle t h\rangle$ is ample (here we allow $t$ to be irrational extending the notation in the obvious way, and observe that although the original Bloch Gieseker result is not stated for twisted bundles the same proof works in this setting, see \cite[p113]{Lazbook2} or \S\ref{sec:Kahler}).  Expanding this as a polynomial in $t$ this gives 
 $$0\neq \int_X c_{d}(E) + t c_{d-1} (E) h + \cdots + t^d h^d \text{ for all  } t\in \mathbb R_{>0}.$$
Clearly this polynomial is strictly positive for $t\gg 0$, and hence since it is nowhere-vanishing, is strictly positive for all $t>0$.  In particular $\int_{X} c_d(E)\ge 0$ as claimed.

To prove the Proposition, we may assume $e:=\rank(E)\ge \lambda_1$ else $s_{\lambda}(E)=0$ and the statement is trivial.  When $|\lambda|=d$,  \eqref{eq:pushforwasschur:twisted} gives a map $\pi:C\to X$ from an irreducible variety $C$ of dimension $n$ and a nef $\mathbb Q$-twisted bundle $U$ of rank $n$ so that $\pi_* c_{n}(U) = s_{\lambda}(E)$.  So by the previous paragraph $\int_X s_{\lambda}(E) = \int_C c_{n}(U)\ge 0.$

Finally suppose $i\ge 0$ and $|\lambda|=d+i$.   Set $\hat{X} = X\times \mathbb P^i$ and $\tau = c_1(\mathcal O_{\mathbb P^1}(1))$.    Since $|\lambda| = \dim(\hat{X})$ we have
$$0\le \int_{\hat{X}} s_{\lambda} (E\langle \tau\rangle) = \int_{\hat{X}} \sum_{j=0}^{|\lambda|+i} s_{\lambda}^{(j)}(E)\tau^j =\int_X s_{\lambda}^{(i)}(E) \int_{\mathbb P^i} \tau^i = \int_X s_{\lambda}^{(i)}(E).$$
\end{proof}

\begin{corollary}\label{cor:Fultonlazdoubleintersection:weak}
Let $X$ be smooth and projective of dimension $d$,  $\lambda$ be a partition of length $d+i-2$, let $E$ be a nef $\mathbb Q$-twisted bundle of rank $e\ge \lambda_1$ and  $h$ be an ample class on $X$.    Then $\int_{X} s_{\lambda}^{(i)}(E) h^2 \ge0$.
\end{corollary}
\begin{proof}
Rescale so $h$ is very ample, and  apply the previous theorem to the restriction of $E$ to the intersection of two general elements in the linear series defined by $h$.
\end{proof}

\begin{remark} [Derived Schur Polynomials are Numerically Positive] \label{rem:derivedschurarelinercombinationsofschur} If $|\lambda| = d+i$ then by taking a resolution of singularities, we have $\int_X s_{\lambda}^{(i)}(E) \ge 0$ for all nef vector bundles $E$ on any irreducible projective variety $X$ of dimension $d$.   That is, $s_{\lambda}^{(i)}$ is a numerically positive polynomial in the sense of Fulton-Lazarsfeld, and hence by their main result \cite[Theorem I]{FultonLazarsfeld} we deduce  $s_{\lambda}^{(i)}$ can be written as a non-negative linear combination of the Schur polynomials $\{s_{\mu} : |\mu|=d\}$.  This answers a question of Xiao \cite[p10]{XiaoHighDegree}.
\end{remark}

\begin{remark}[Monomials of Derived Schur Classes]\label{rem:monomialsofderivedschurclasses} It is easy to extend this to monomials of derived Schur polynomials.  That is, if $E_1,\ldots,E_p$ are nef bundles on $X$ and $\lambda^{1},\ldots,\lambda^{p}$ are partitions such that $\sum_j |\lambda^{j}| = d$ then \begin{equation}\label{eq:monomialschupositive} \int_X \prod_j s_{\lambda^{j}}(E_j)\ge 0.\end{equation} 
We simply repeat the
proof of Proposition \ref{prop:fultonlazderivedshur}
using \eqref{eq:productschur} in place of \eqref{eq:pushforwasschur:twisted}).  For the derived case suppose we also have integers $i_1,\ldots,i_p$ and that our partitions are such that $\sum_j |\lambda^{(j)}| - i_j = d$.  Then
\begin{equation}\int_X \prod_j s_{\lambda^{j}}^{(i_j)}(E_j) \ge 0.\label{eq:monomialsofderivedarenonnegative}\end{equation}
To see this consider the product $\hat{X}:=X\times \prod_j \mathbb P^{i_j}$ and let $\tau_j$ be the pullback of the hyperplane class in $\mathbb P^{i_j}$ to $\hat{X}$.  Then \eqref{eq:monomialschupositive} applies to the class $\prod_j s_{\lambda^{j}}(E_j(\tau_j))$.  Expanding this as a symmetric polynomial in the $\tau_j$ the coefficient of $\prod_j \tau_j^{i_j}$ is precisely $\prod_j s_{\lambda^{j}}^{(i_j)}(E_j)$ so \eqref{eq:monomialsofderivedarenonnegative} follows.  The analog of  Corollary \ref{cor:Fultonlazdoubleintersection:weak} also holds for monomials of derived Schur polynomials.  \end{remark}

\section{Hodge-Riemann classes}\label{sec:HRclasses:definitions}

Let $X$ be a  projective smooth variety dimension $d$ and let $\Omega\in H^{d-2,d-2}(X;\mathbb R)$.  This defines an intersection form
$$ Q_{\Omega}(\alpha,\alpha')= \int_{X} \alpha \Omega \alpha' \text{ for } \alpha,\alpha'\in H^{1,1}(X;\mathbb R).$$

\begin{definition}[Hodge-Riemann Property]
We say that a bilinear form $Q$ on a finite dimensional vector space has the \emph{Hodge-Riemann property} if $Q$ is non-degenerate and has precisely one positive eigenvalue.    We say that $\Omega\in H^{d-2,d-2}(X;\mathbb R)$ has the Hodge-Riemann property if $Q_{\Omega}$ does, and denote by $\HR(X)$ denote the set of all $\Omega$ with this property.
\end{definition}

\begin{definition}[Weak Hodge-Riemann Property]
A bilinear form $Q$ on a finite dimensional vector space is said to have the \emph{weak Hodge-Riemann property} if it is a limit of bilinear forms that have the Hodge-Riemann property.  We say that $\Omega$ has the weak Hodge-Riemann property if $Q_{\Omega}$ does, and denotes by $\wHR(X)$ the set of $\Omega$ with this property.
\end{definition}

So $Q$ has the weak Hodge-Riemann property if and only if if has one eigenvalue that is non-negative, and all the others are non-positive.     Clearly  
$$\overline{\HR}(X)\subset \wHR(X)$$
but we do not claim these are equal (the issue being that in principle $Q_{\Omega}$ could be the limit of bilinear forms with the Hodge-Riemann property that do not come from classes in $H^{d-2,d-2}(X;\mathbb R)$).      If $h$ is ample then by the classical Hodge-Riemann bilinear relations $h^{d-2}\in \HR(X)$, and so $\wHR(X)$ is a  non-empty closed cone inside $H^{d-2,d-2}(X;\mathbb R)$.\medskip

It is convenient to work with $\wHR(X)$ as it behaves well with respect to pullbacks and pushforwards.  This is captured by the following simple piece of linear algebra.

\begin{lemma}\label{lem:simplelinearalgebra}
Let $f:V\to W$ be a linear map of vector spaces and $Q_V$ and $Q_W$ be bilinear forms on $V$ and $W$ respectively such that $$Q_W(f(v),f(v')) = Q_V(v,v') \text{ for all } v,v'\in V.$$
Suppose that $Q_W$ has the weak Hodge-Riemann property and there is a $v_0\in V\setminus\{0\}$  with $Q_V(v_0,v_0)\ge 0$.  Then $Q_V$ has the weak Hodge-Riemann property.
\end{lemma}
\begin{proof}
Let $N = \operatorname{ker}(f)$.  Then $N$ is orthogonal to all of $V$ with respect to $Q_V$.   The signature on a complementary subspace to $N$ is induced by $Q_W$.  Thus $Q_{V}$ can only be negative semi-definite, or have the weak Hodge-Riemann property, and the assumption that $Q_V(v_0,v_0)\ge 0$ means it is the latter case that occurs. 
\end{proof}

\begin{lemma}[Pullbacks]\label{lem:pullbacksgiveHR}
Let $\pi:X'\to X$ be a surjective map between smooth varieties of dimension $d$.  Let $\Omega\in H^{d-2,d-2}(X,\mathbb R)$ and suppose there is an $h\in H^{1,1}(X;\mathbb R)\setminus\{0\}$ with $\int_X \Omega h^2\ge 0$ and that $\pi^* \Omega \in {\wHR}(X')$.    Then $\Omega \in {\wHR}(X)$.
\end{lemma}
\begin{proof}
 This follows from Lemma \ref{lem:simplelinearalgebra} applied to $\pi^*: H^{1,1}(X;\mathbb R)\to H^{1,1}(X';\mathbb R)$ since
 $Q_{\pi^* \Omega}(\pi^*\alpha,\pi^* \alpha') = \int_{X'} \pi^*( \Omega \alpha \alpha') =\deg(\pi) \int_X \Omega \alpha\alpha' = \deg(\pi) Q_{\Omega}(\alpha,\alpha')$.
\end{proof}

\begin{lemma}[Pushforwards]\label{lem:HRpushforward}
Let $\pi:X'\to X$ be a surjective map between smooth varieties.   Let $\Omega'\in {\wHR}(X')$ and suppose there is an $h\in H^{1,1}(X;\mathbb R)\setminus\{0\}$  with $\int_X (\pi_*\Omega') h^2\ge 0$.   Then $\pi_* \Omega' \in {\wHR}(X)$.
\end{lemma}
\begin{proof}
This follows from Lemma \ref{lem:simplelinearalgebra} applied to $\pi^*: H^{1,1}(X;\mathbb R)\to H^{1,1}(X';\mathbb R )$ since from the projection formula,  $$Q_{\Omega'}(\pi^*\alpha,\pi^*\alpha') = \int_{X'} \Omega' (\pi^* \alpha) (\pi^*\alpha') = \int_X \pi_*\Omega' \alpha \alpha' = Q_{\pi_*\Omega}(\alpha,\alpha').$$
\end{proof}

We will need the following variant that allows for an intermediate space that might not be smooth.

\begin{lemma}\label{lem:intermediatesingular}
Let $X,Y,Z$ be irreducible projective varieties with morphisms $Z\stackrel{\sigma}{\to} Y \stackrel{\pi}{\to} X$ and assume that $Z$ and $X$ are smooth.  Let $d=\dim X$ and assume $Z$ and $Y$ are of the same dimension $n$ and that $\sigma$ is surjective.  Let $\Omega\in H^{2n-4}(Y;\mathbb R)$ be such that $\Omega': = \pi_* \Omega\in H^{d-2,d-2}(X;\mathbb R)$.  Assume 
\begin{enumerate}
\item $\sigma^*\Omega \in \wHR(Z)$.
\item There exists an $h\in H^{1,1}(X;\mathbb R)\setminus\{0\}$ such that $\int_X (\pi_*\Omega) h^2 \ge 0$.
\end{enumerate}
Then $\pi_*\Omega\in \wHR(X)$.
\end{lemma}
\begin{proof}
Let $p= \pi\circ \sigma:Z\to X$.  By the projection formula
\begin{align*}
Q_{\sigma^* \Omega}(p^*\alpha,p^*\alpha') &= \int_Z \sigma^* \Omega p^*\alpha p^*\alpha' = \int_Z \sigma^* \Omega \sigma^* \pi^* \alpha \sigma^* \pi^*\alpha' \\
&= \deg(\sigma) \int_Y \Omega \pi^*\alpha \pi^*\alpha' = \deg(\sigma) \int_Z (\pi_* \Omega) \alpha \alpha' = \deg(\sigma) Q_{\pi_* \Omega}( \alpha,\alpha').\end{align*}
Thus the result  follows from Lemma \ref{lem:simplelinearalgebra} applied to $p^*:H^{1,1}(X;\mathbb R)\to H^{1,1}(Z;\mathbb R)$.
\end{proof}

\section{Schur classes are in $\overline{HR}$}\label{sec:schurclassesareinHR}

\begin{lemma}\label{lemma:cn2I}
Let $X$ be a smooth projective manifold of dimension $d\ge 4$, and $E$ be a nef $\mathbb Q$-twisted bundle of rank $d-2$.   Then $c_{d-2}(E)\in {\wHR}(X)$.
\end{lemma}
\begin{proof}
This is exactly as in \cite[Proposition 3.1]{RossTomaHR}.  First assume that $E$ is ample and $X$ is smooth.  By a consequence of the Bloch-Gieseker Theorem  for all $t\in \mathbb R_{\ge 0}$  the intersection form  
$$Q_{t}(\alpha):= \int_X \alpha c_{d-2} (E\langle th \rangle) \alpha \text{ for } \alpha\in H^{1,1}(X;\mathbb R)$$ is non-degenerate (we remark that we are allowing possibly irrational $t$ here, and then $c_{d-2}(E\langle th \rangle)$ is to be understood as being defined as in \eqref{eq:defcherntwisted}).    Now for small $t$ we have $$ c_{d-2}(E\langle th\rangle) = t^{d-2} h^{d-2} + O(t^{d-3}).$$
Observe that for an intersection form $Q$, having signature $(+,-\ldots,-)$ is invariant under multiplying $Q$ by a positive multiple, and is an  open condition as $Q$ varies continuously.  Thus since we know that $h^{d-2}$ has the Hodge-Riemann property, the intersection form $(\alpha,\beta)\mapsto \int_X \alpha h^{d-2} \beta$ has signature $(+,-\ldots,-)$, and hence so does $Q_t$ for $t$ sufficiently large.   But $Q_t$ is non-degenerate for all $t\ge 0$, and hence $Q_t$ must have this same signature for all $t\ge 0$.  Thus $c_{d-2}(E)\in \HR(X)$.  

Since any $\mathbb Q$-twisted nef bundle $E$ can be approximated by an $\mathbb Q$-twisted ample vector bundle we deduce that $c_{d-2}(E)\in \overline{\HR}(X)\subset \wHR(X)$.  
\end{proof}

\begin{theorem}[Derived Schur Classes are in $\overline{\HR}$]\label{thm:derivedschurclassesareinHR}
Let $X$ be smooth and projective of dimension $d\ge 2$, let $\lambda$ be a partition of length $d+i-2$ and let $E$ be a $\mathbb Q$-twisted nef vector bundle on $X$.  Then
$$s_{\lambda}^{(i)}(E) \in \overline{\HR}(X).$$
\end{theorem}
\begin{proof}
The statement is trivial unless $e:=\rank(E)\ge \lambda_1$ and $d\ge 2$ which we assume is the case.  When $d=3$,  $s_{\lambda}^{(i)}$ is a positive multiple of $c_1$ and then the result we want follows from the classical Hodge-Riemann bilinear relations.  So we can assume from now on that $d\ge 4$.   

Fix an ample class $h$ on $X$.  We first prove that $s_{\lambda}(E)\in \wHR(X)$.  Consider the case $i=0$ so $|\lambda| = d-2$.   By Corollary \ref{cor:Fultonlazdoubleintersection:weak} $\int_X s_{\lambda}(E) h^2\ge 0$.  Also, the cone construction  described in  \S\ref{sec:coneclasses} (particularly  \eqref{eq:pushforwasschur:twisted}) gives an irreducible variety $\pi:C\to X$ of dimension $n$ and a nef $\mathbb Q$-twisted vector bundle $U$ of rank $n-2$ such that $$\pi_* c_{n-2}(U) = s_{\lambda}(E).$$

Since $C$ is irreducible we can take a resolution of singularities $\sigma:C'\to C$.  Then $\sigma^*U$ is also nef, and Lemma \ref{lemma:cn2I} gives  $c_{n-2}(\sigma^*U)\in \wHR(C')$.  Thus Lemma \ref{lem:intermediatesingular} implies $s_{\lambda}(E)\in {\wHR}(X)$.

Consider next the case  $i\ge 1$, so  $|\lambda| = d+i-2$.  Again by Corollary \ref{cor:Fultonlazdoubleintersection:weak}, $\int_X s_{\lambda}^{(i)}(E) h^2\ge 0$.  Consider the product $\hat{X} = X\times \mathbb P^i$ and set $\tau = c_1(\mathcal O_{\mathbb P^i}(1))$.   Suppressing pullback notation, the $\mathbb Q$-twisted bundle $E\langle \tau \rangle$ on $\hat{X}$ is nef, so by the previous paragraph $s_{\lambda}(E\langle \tau\rangle) \in \wHR(\hat{X})$.   Now
$$s_{\lambda}(E\langle \tau\rangle) = \sum_{j=0}^{|\lambda|} s_{\lambda}^{(j)}(E) \tau^j$$
so if $\pi:\hat{X}\to X$ is the projection
$$\pi_* s_{\lambda}(E\langle \tau\rangle) = s_{\lambda}^{(i)}(E).$$
Thus by Lemma \ref{lem:HRpushforward} we get also $s_{\lambda}^{(i)}(E)\in \wHR(X)$.



To complete the proof define
$$\Omega_t = s_{\lambda}^{(i)}(E\langle th \rangle) \text{ for } t\in \mathbb Q_{\ge 0}$$
and
$$ f(t) = \det(Q_{\Omega_t}).$$
Note that the leading term of $\Omega_t$ is a positive multiple of $h^{d-2}$ (this is Example \ref{ex:lowestdegree} and it is here we use that $e\ge\lambda_1$).  In particular, for $t$ sufficiently large $Q_{\Omega_t}$ is non-degenerate (in fact it has the Hodge-Riemann property).  Thus $f$ is not identically zero, and since it is a polynomial in $t$ this implies $f(t)\neq 0$ for all but finitely many $t$.  Thus there is an $\epsilon>0$ so that $f(t)\neq 0$ for rational $0<t<\epsilon$ and we henceforth consider only $t$ in this range.  Then $Q_{\Omega_t}$ is non-degenerate, and as $Q_{\Omega_t}(h,h)\ge 0$ it cannot be negative definite.  The previous paragraph gives $\Omega_t\in {\wHR}$, so we must actually have $\Omega_t \in \HR(X)$ for small $t\in \mathbb Q_{>0}$.  Thus $\Omega_0  = s_{\lambda}^{(i)}(E)\in \overline{\HR}(X)$ as claimed.
\end{proof}

\begin{remark}\label{rmk:comparisonwitholdpaper}
Note the above proof gives more, namely that if $h$ is an ample class and $E$ is nef and $\lambda_1\le \rank(E)$ we have
 $$s_{\lambda}^{(i)}(E\langle th\rangle)\in \HR(X) \text{ for all but possibly finitely many } t\in \mathbb Q_{>0}.$$ 
As mentioned in the introduction, the main result of \cite{RossTomaHR} says more namely that if $E$ is ample of rank at least $\lambda_1$ then $s_{\lambda}^{(i)}(E)\in \HR(X)$, but the proof of that statement is significantly harder.
\end{remark}

\begin{theorem}[Monomials of Schur Classes are in $\overline{\HR}$]\label{thm:monomialsderivedschurclassesareinHR}
Let $X$ be smooth and projective of dimension $d$ and $E_1,\ldots,E_p$ be nef vector bundles on $X$. 
   Let $\lambda^1,\ldots,\lambda^p$ be partitions such that $$\sum_i |\lambda^i| = d-2.$$ Then the monomial of  Schur polynomials
$$\prod_i s_{\lambda^i}(E_i)$$
lies in $\overline{\HR}(X)$.
\end{theorem}
\begin{proof}
The proof is similar to what has already been said, so we merely sketch the details.    Set $\Omega=\prod_i s_{\lambda^i}(E_i)$.  Then \eqref{eq:productschur} gives a map $\pi:C\to X$ from an irreducible variety of dimension $n$ and nef bundle bundle $U$ on $C$ so $\pi_* c_{n-2}(U) = \Omega$.   A small modification of the proof of Proposition \ref{prop:fultonlazderivedshur} and 
Corollary \ref{cor:Fultonlazdoubleintersection:weak} means that if $h$ is ample  $\int_X \Omega h^2\ge 0$.

Consider
$$\Omega_t : = \pi_* c_{n-2}(U\langle t \pi^*h \rangle)$$
and take a resolution $\sigma:C'\to C$.  Then $\sigma^* U\langle \pi^*h\rangle$ remains nef, so Lemma \ref{lem:intermediatesingular} implies $\Omega_t \in \wHR(X)$.  

Now we can equally apply this construction replacing each $E_i$ with $E_i\otimes \mathcal O(th)$ for $t\in \mathbb N$ (which one can check does not change $\pi:C\to X$) giving
$$\pi_* c_{n-2} ( U \langle th\rangle) = \prod_i s_{\lambda^i}(E_i\langle th\rangle) \text{ for } t\in \mathbb N.$$
In particular applying Example \ref{ex:lowestdegree} to each factor on the right hand side,  the highest power of $t$ is a positive multiple of $h^{d-2}$.     Thus for almost all $t\in \mathbb Q_{>0}$ we have $Q_{\Omega_t}$ is non-degenerate, and so in fact $Q_{\Omega_t}\in \HR(X)$.  Taking the limit as $t\to 0$ gives the result we want.
\end{proof}

\section{The K\"ahler case} \label{sec:Kahler}

The main place in which projectivity has been used so far is in the application of the Bloch-Gieseker Theorem, and here we explain how this projectivity assumption can be relaxed.   Following Demailly-Peternell-Schneider \cite{DemaillyPeternellSchneider} we say a line bundle $L$ on a compact K\"ahler manifold $X$ is \emph{nef} if for all $\epsilon>0$ and all K\"ahler forms $\omega$ on $X$ there exists a hermitian metric $h$ on $L$ with curvature $dd^c \log h \ge -\epsilon \omega$.   We say that a vector bundle $E$ on $X$ is nef if the hyperplane bundle $\mathcal O_{\mathbb P(E)}(1)$ is nef.   

For the rest of this section let $(X,\omega)$ be a compact K\"ahler manifold of dimension $d$.   Given a vector bundle $E$ and $\delta\in H^{1,1}(X;\mathbb R)$ we can consider the $\mathbb R$-twisted bundle $E\langle \delta \rangle$ whose Chern classes are defined just as in the case of $\mathbb Q$-twists in the projective case.   We identify $\mathbb P(E\langle \delta\rangle)$ with $\mathbb P(E)$, and say that $E\langle \delta\rangle$ is nef if for any K\"ahler metric $\omega'$ on $\mathbb P(E)$, any $\epsilon>0$, and any closed $(1,1)$ form $\delta'$ on $X$ such that $[\delta']=\delta$,   there exists a hermitian metric $h$ on $\mathcal O_{\mathbb P(E)}(1)$ such that
$$dd^c \log h + \pi^* \delta' \ge -\epsilon \omega'.$$
We refer the reader to \cite{DemaillyPeternellSchneider} for the fundamental properties of nef bundles on compact K\"ahler manifolds, in particular to the statement that a quotient of a nef bundle is again nef, and the direct sum of two nef bundles is again nef (and each of these statements extend to the case of $\mathbb R$-twisted nef bundles with minor modifications of the proofs involved).

\begin{theorem}[Bloch-Gieseker for K\"ahler Manifolds]
Let $E$ be a nef $\mathbb R$-twisted vector bundle of rank $e\le d$ and $t>0$.    Let $e+j\le d$ and consider
$$\Omega := c_{e}(E\langle t\omega\rangle) \wedge \omega^j.$$
Then then map
$$H^{d-e-j}(X) \stackrel{\wedge \Omega}{\longrightarrow} H^{d+e+j}(X)$$
is an isomorphism.
\end{theorem}
\begin{proof}
Write $E = E'\langle \delta\rangle$ where $E'$ is a genuine vector bundle.  Fix $t>0$ and set $E_t: = E\langle t\omega\rangle = E'\langle \delta + t\omega\rangle$.    Set $\pi:\mathbb P(E')\to X$ and define $\zeta' = c_1(\mathcal O_{\mathbb P(E')}(1))$ and $\zeta: = \zeta' + \pi^*(\delta+ t [\omega])$.  Then $\zeta^e - c_1(E_t) \zeta^{e-1} + \cdots + (-1)^e c_e(E_t)=0$ where we supress pullback notation for convenience.

Suppose $a\in H^{d-e-j}(X)$ has $a c_e(E_t)\omega^j=0$, and we will show that $a=0$.  To this end define
$$ b  = a.(\zeta^{e-1}- c_1(E_t) \zeta^{e-2} + \cdots +(-1)^{e-1} c_{e-1}(E_t))$$
so by construction
$$ \zeta b \omega^j = \pm a c_e(E_t) \omega^j=0$$
We claim that $\zeta$ is a K\"ahler class.   Given this for now, the Hard-Lefschetz property for $\zeta$ then gives $b\omega^j=0$ and hence $a\omega^j = \pi_* (b\omega^j)=0$ and hence $a=0$ by the Hard-Lefschetz property of $\omega^j$

It remains to show that $\zeta$ is K\"ahler, and the following is essentially what is described in \cite[proof of Theorem 1.12]{DemaillyPeternellSchneider}.    Fix $\omega'$ a K\"ahler metric on $\mathbb P(E')$, and fix a hermitian metric on $E'$ which induces a hermitian metric $\hat{h}$ on $\mathcal O_{\mathbb P(E')}(1)$.  Then $dd^c\log \hat{h}$ is strictly positive in the fiber directions, so there is a constant $C>0$ with
$$ dd^c\log\hat{h} + C \pi^* \omega \ge C^{-1}\omega'.$$
Let $\delta'$ be a closed $(1,1)$-form on $X$ with $[\delta']=\delta$, and choose $\epsilon>0$  sufficiently small that $(t-C^2 \epsilon)\omega +C\epsilon \delta'>0$.  Then as $E$ is assumed to be nef there is a hermitian metric $h$ on $\mathcal O_{\mathbb P(E')}(1)$ such that $dd^c \log h + \pi^* \delta'\ge -\epsilon \omega'$.   

Then the class $\zeta = c_1(\mathcal O_{\mathbb P(E')}(1)) + \pi^*[\delta+t\omega]$ is represented by the form
$$ (1-C\epsilon) dd^c \log h + C\epsilon dd^c \log \hat{h} + \pi^*(\delta' + t\omega) $$
which is bounded from below by
\begin{align*}
(1-C\epsilon) (-\epsilon \omega' - \pi^*\delta') &+ C\epsilon(C^{-1}\omega' - C\pi^*\omega) + \pi^*(t\omega + \delta')\\
&= C\epsilon^2 \omega' + (t-C^2\epsilon) \pi^*\omega  + C\epsilon \pi^*\delta'\\
& \ge C\epsilon^2 \omega' >0.
\end{align*}
Thus $\zeta$ is a K\"ahler class as claimed.
\end{proof}

\begin{corollary}
Let $E$ be a nef $\mathbb R$-twisted vector bundle of rank $e\le d$ and $j=d-e$.  Then
$$\int_X c_e(E)\omega^j \ge 0$$
\end{corollary}
\begin{proof}
Let $f(t) = \int_X c_e(E\langle t\omega\rangle)\omega^j$.  The Bloch-Gieseker theorem implies $f(t)\neq 0$ for all $t>0$, and since it is clearly positive for $t\gg 0$ $f$ is not identically zero.  Since $f$ is polynomial in $t$ we get $f(t)>0$ for $t>0$ sufficiently small, which proves the statement.
\end{proof}

From here almost all the results in this paper extend to the K\"ahler case, and the proofs have only trivial modifications.  We state only one and leave the rest to the reader.

\begin{theorem}[Derived Schur classes of nef vector bundles on K\"ahler manifolds are in $\overline{\HR}$]\label{thm:derivedschurclassesareinHR:Kahler}
Let $X$ be  a compact K\"ahler manifold of dimension $d\ge 2$, let $\lambda$ be a partition of length $d+i-2$ and let $E$ be an $\mathbb R$-twisted nef vector bundle on $X$.  Then
$$s_{\lambda}^{(i)}(E) \in \overline{\HR}(X).$$
\end{theorem}

\section{Combinations of Derived Schur Classes}\label{sec:combinations}

An interesting feature of the Hodge-Riemann property for bilinear forms is that it generally is not preserved by taking convex combinations, and so there is no reason to expect that a convex combination of classes with the Hodge-Riemann property again has the Hodge-Riemann property.  In fact this is true even for combinations of Schur classes of an ample vector bundle as the following example shows

\begin{example}[\!\!\protect{\cite[Section 9.2]{RossTomaHR}}]\label{example:convexcombinations}
Let $X=\mathbb P^2\times \mathbb P^3$  Then $N^1(X)$ is two-dimensional, with generators $a,b$ that satisfy $a^3=0$, $a^2b^3=1$.
Set
$\mathcal O_X(a,b) = \mathcal O_{\mathbb P_2}(a) \boxtimes \mathcal O_{\mathbb P^3}(b)$ and consider the nef vector bundle 
$$ E = \mathcal O(1,0) \oplus \mathcal O(1,0) \oplus \mathcal O(0,1).$$
One computes that the form
$$ (1-t) c_3(E) + t s_{(1,1,1)}(E)$$
gives an intersection form on $N^1(X)$ with matrix 
$$Q_t:= \left(\begin{array}{cc} t&2t \\ 2t&1+2t \end{array}\right).$$
For $t\in (0,1/2)$ the matrix $Q_t$ has two strictly positive eigenvalues.   Thus fixing $t\in (0,1/2)$, any small pertubation of $E$ by an ample class gives an ample $\mathbb Q$-twisted bundle $E'$ so that $(1-t)c_3(E') + t s_{(1,1,1)}(E')$ does not have the Hodge-Riemann property.
 \end{example}
 
Given this it is interesting to ask if there are particular convex combinations of (derived) Schur classes that do retain the Hodge-Riemann property.    To state one such result we need the following definition, for which we recall a matrix is said to be \emph{totally positive} if all its minors have non-negative determinant,.

\begin{definition}[P\'olya Frequency Sequence]
Let $\mu_0,\ldots,\mu_{N}$ be non-negative numbers, and set $\mu_i=0$ for $i<0$.  We say $\mu_0,\ldots,\mu_N$ is a \emph{P\'olya frequency sequence} if the matrix
$$ \mu:=(\mu_{i-j})_{i,j=0}^N$$
is  totally positive.\end{definition}

\begin{theorem}\label{thm:poyla}
Suppose that $X$ has dimension $d\ge 4$ that $h$ is an nef class on $X$ and $E$ is a nef vector bundle.  Let  $|\lambda|=d-2$ and $\mu_0,\ldots,\mu_{d-2}$ be a P\'olya frequency sequence.   Then the class
\begin{equation}\sum_{i=0}^{d-2} \mu_i s_{\lambda}^{(i)} (E) h^{i}\label{eq:Poylaclass}\end{equation}
lies in $\overline{\HR}(X)$.
\end{theorem}

 Theorem \ref{thm:poyla} follows quickly from the following statement, for which we recall $c_i$ denotes the $i$-th elementary symmetric polynomial.
 
 \begin{proposition}\label{prop:poyla:divisor}
Suppose that $X$ has dimension $d\ge 4$ and $E$ is a nef  vector bundle.   Let $\lambda$ be a partition of $d-2$.  Let $D_1,\ldots,D_q$ be ample $\mathbb Q$-divisors on $X$ for some $q\ge 1$.     Then for any $t_1,\ldots,t_q\in \mathbb Q_{>0}$ the class
$$\sum_{i=0}^{d-2} s_{\lambda}^{(i)}(E) c_i(t_1 D_1,\ldots,t_q D_q)$$
lies in $\overline{\HR}(X)$
\end{proposition}

\begin{proof}[Proof of Theorem \ref{thm:poyla}]
If all the $\mu_i$ vanish the statement is trivial, so we assume this is not the case. From the Aissen-Schoenberg-Whitney Theorem \cite{AissenSchoenbergWhitney}, the assumption that $\mu_i$ is a P\'olya frequency sequence implies that the generating function
$$\sum_{i=0}^{d-2} \mu_i z^i$$
has only real roots, and since each $\mu_i$ is non-negative these roots are then necessarily non-positive.   Writing these roots as $\{-t_j\}$ for $t_j\in \mathbb R_{\ge 0}$ means
$$\sum_{i=0}^{d-2} \mu_i z^i = \kappa\prod_{j=0}^N (z+t_j) \text{ where }\kappa>0$$
which implies 
$$ \mu_i =  \kappa c_i(t_1,\ldots,t_{ N}) \text{ for all }i.$$

Now for each $j$ let $t_j^{(n)}\in \mathbb Q_{>0}$ tend to $t_j$ as $n\to \infty$.   Fix an ample divisor $h''$ and consider the class $h': = h + \frac{1}{n} h''$.   Proposition \ref{prop:poyla:divisor} (applied with $q= N$ and $D_1=\cdots = D_q = h'$) implies
$$\sum_{i=0}^{d-2} s_{\lambda}^{(i)}(E) c_i(t^{(n)}_1,\ldots,t^{(n)}_{ N}) (h')^{i}$$
lies in $\overline{\HR}(X)$.     Taking the limit as $n\to \infty$ gives the statement we want.
\end{proof}

\begin{proof}[Proof of Proposition \ref{prop:poyla:divisor}]
Set
$$ \Omega : = \Omega(D_1,\ldots,D_p): = \sum_{i=0}^{d-2} s_{\lambda}^{(i)}(E) c_i(D_1,\ldots,D_p).$$
Without loss of generality we may assume all the $D_i$ are integral and very ample.  Write $t_j = r_j/s$ for some positive integers $r_j$ and $s$.  By an iterated application of the Bloch-Gieseker covering construction, we find a finite $u:Y\to X$ and line bundles $\eta_j$ on $X'$ such that that $\eta_j^{\otimes s} = u^* \mathcal O(D_j)$.  Thus 
$$ r_j c_1(\eta_j) = t_j u^* D_j.$$

Set $E' = u^* E$.  Consider the cone construction for $E'$ as described in \S\ref{sec:coneclasses}.   That is, there is a surjective $\pi: C\to Y$ from an irreducible variety $C$ of dimension $n$, and a nef vector bundle $U$ on $C'$ of rank $n-2$ such that $\pi_* c_{n-2}(U)= s_{\lambda}(E')$.  In fact more is true namely;

\begin{lemma}\label{lemma:pushforwardderivedschur:latex}
\begin{equation}\pi_* c_{n-2-i}(U|_C) = s_{\lambda}^{(i)}(E') \text{ for } 0\le i\le |\lambda|.\label{eq:pushforwardderivedschur:removed}\end{equation}
\end{lemma}
\begin{sketchproof}
Formally this is clear: for if $\delta'\in H^{1,1}(X;\mathbb R)$ then $c_{n-2}(U\langle\pi^* \delta'\rangle)=  \sum c_{n-2-i}(U) (\pi^*\delta')^i$ and pushing this forward to $X$ gives a polynomial in $\delta'$ of classes on $X$ whose coefficients are the derived Schur classes $s_{\lambda}^{(i)}(E')$.  For a full proof we refer the reader to  \cite[Proposition 5.2]{RossTomaHR}.
\end{sketchproof}

Continuing with the proof of the Proposition, set
$$ F = \bigoplus_{i=1}^p \eta_i^{\otimes r_i}$$
so
$$ c_j(F) = c_j ( r_1 c_1(\eta_1), \cdots, r_p c_1(\eta_p)) = u^* c_j( t_1D_1,\ldots, t_p D_p).$$
Then on $C'$ the bundle
$$ \tilde{U}  := U \oplus \pi^* F$$ 
is nef.  Take a resolution $\sigma:C\to C'$, the vector bundle $\sigma^*U$ remains nef and so using Theorem \ref{thm:derivedschurclassesareinHR} and Lemma \ref{lem:intermediatesingular} 
$$\pi_* c_{n-2}(\tilde{U}) \in \wHR(Y)$$
But
\begin{align*}
\pi_* c_{n-2}(\tilde{U})&= \pi_*( c_{n-2}(U) + c_{n-3}(U) \pi^* c_1(F) + \cdots + c_{n-2-d}(U) \pi^* c_d(F))\\
&= s_{\lambda}(E') + s_{\lambda}^{(1)}(E')c_1(F) + \cdots + s_{\lambda}^{(d-2)}(E') c_{d-2}(F) \\
&= u^*\Omega.
\end{align*}
So by Lemma \ref{lem:pullbacksgiveHR} applied to $u:Y\to X$ we conclude that $\Omega\in \wHR(X)$.

To show that in fact $\Omega\in \overline{\HR}(X)$ we consider the effect of replacing each $D_i$ with $D_i +th$.  Let $\Omega_t: = \Omega(D_1+th,\ldots,D_p+th)$ which is a polynomial in $t$ whose $t^{d-2}$ term is some positive multiple of $h^{d-2}$.   Setting $f(t) = \det(Q_{\Omega_t})$ we conclude exactly as in the end of the proof of Theorem 
\ref{thm:derivedschurclassesareinHR} that $\Omega_t\in \HR(X)$ for  $t\in \mathbb Q_{+}$ sufficiently small, and thus $\Omega\in \overline{\HR}(X)$ as required.
\end{proof}

\begin{question}
Suppose that $\mu_1,\ldots,\mu_{d-2}$ is a P\'olya frequency sequence with each $\mu_i$ strictly positive,  and that $h$ and $E$ are ample.    Is it then the case that the class in \eqref{eq:Poylaclass} is actually in $\HR(X)$?    The difficulty here is that to follow the proof we have given above one needs to address the possibility that some of the $t_j$ are irrational.  
\end{question}

\section{Inequalities}\label{sec:inequalities}

\subsection{Hodge-Index Type inequalities}
The simplest and most fundamental inequality obtained from the Hodge-Riemann property is the Hodge-index inequality.   


\begin{theorem}[Hodge-Index Theorem]
Let $X$ be a manifold of dimension $d$ and $\Omega\in \wHR(X)$. 
If $\beta \in H^{1,1}(X)$ is such that $\int_X \beta^2 \Omega\ge0$ then  for any $\alpha \in H^{1,1}(X)$ it holds that
\begin{equation} \int_X \alpha^2 \Omega \int_X \beta^2 \Omega \le \left( \int_X \alpha \beta \Omega\right)^2\label{eq:HI}\end{equation}
Moreover if $\Omega\in \HR(X)$ and $\int_X \beta^2 \Omega>0$ then equality holds in \eqref{eq:HI} if and only if $\alpha$ and $\beta$ are proportional.
\end{theorem}
\begin{proof}
The statement is about symmetric bilinear forms with the given signature and its proof is standard.  Indeed, the case when $\int_X \beta^2 \Omega=0$ is trivial and the case when the intersection form is nondegenerate and $\int_X \beta^2 \Omega>0$ is classical. Finally, the case when the intersection form is degenerate and $\int_X \beta^2 \Omega>0$ reduces itself to the previous one by modding out the kernel of the intersection form. 
\end{proof}

In particular (namely Theorem \ref{thm:derivedschurclassesareinHR})  the inequality \eqref{eq:HI} applies when $\Omega=s_{\lambda}(E)$ whenever $\lambda$ is a partition of $d-2$, $E$ is a nef $\mathbb Q$-twisted bundle on $X$ and $\beta$ is nef.      We now prove a variant of this that gives additional information.

\begin{theorem}\label{thm:thmschurhodgeimproved}
Let $X$ be a projective manifold of dimension $d\ge 4$ and let $E$ be a $\mathbb Q$-twisted nef vector bundle and $h\in H^{1,1}(X;\mathbb R)$ be nef.  Also let $\lambda$ be a partition of length $|\lambda|= d-1$.    Then for all $\alpha\in H^{1,1}(X;\mathbb R)$,
\begin{equation}\int_X  \alpha^2 s_{\lambda}^{(1)}(E) \int_X h s_{\lambda}(E)  \le 2 \int_X \alpha h s_{\lambda}^{(1)}(E)   \int_X \alpha s_{\lambda}(E)\label{eq:schurhodgeimproved}\end{equation}
\end{theorem}

\begin{remarks}
\begin{enumerate}
\item In the case that $\lambda = (d-1)$ and $\rank(E)=d-1$ the inequality \eqref{eq:schurhodgeimproved} becomes
\begin{equation}
\int_X \alpha^2c_{d-2}(E)  \int_X  h c_{d-1}(E) \le 2\int_X \alpha h c_{d-2}(E) \int_X \alpha c_{d-1}(E).\label{eq:schurhodgeimproved:chern}
\end{equation}
This was previously proved in \cite[Theorem 8.2]{RossTomaHR}.  In fact \eqref{eq:schurhodgeimproved:chern} was shown to hold for all nef vector bundles of rank at least $d-1$ and if $E,h$ are assumed ample then equality holds in  \eqref{eq:schurhodgeimproved:chern}  if and only if $\alpha=0$.  We imagine a similar statement holds in the context of Theorem \ref{thm:thmschurhodgeimproved}.
\item Assume in the setting of Theorem \ref{thm:thmschurhodgeimproved} that $\int_X s_{\lambda}(E)h >0$ and let $W$ be the kernel of the map $H^{1,1}(X) \to \mathbb R$ given by $\alpha\mapsto \int_X \alpha s_{\lambda}(E)$.   Then $W$ has codimension 1, and \eqref{thm:thmschurhodgeimproved} says that the intersection form $Q_{s_{\lambda}(E)}$
is negative semidefinite on the codimension one subspace
$$  \{ \alpha \in H^{1,1}(X) : \int_X \alpha s_{\lambda}^{(1)}(E) =0\}.$$
This is different information to the Hodge-Index inequality which is essentially a reformulation of the fact that this intersection form is negative semidefinite on the orthogonal complement of $h$.
\item The inequality \eqref{eq:schurhodgeimproved} generalizes to any homogeneous symmetric polynomial $p$ in $e$ variables with the property that $p(E)\in \overline{\HR}(X)$ for all $\mathbb Q$-twisted nef vector bundles $E$ of rank $e$ (with the obvious definition for the derived polynomials $p^{(i)})$.
\end{enumerate}
\end{remarks}

\begin{proof}[Proof of Theorem \ref{thm:thmschurhodgeimproved}]

If $e:=\rank(E)<\lambda_1$ the statement is trivial, so we assume $e\ge \lambda_1$.  We start with some reductions.  By continuity, it is sufficient to prove this under the additional assumption that $h$ is ample.  Also replacing $E$ with $E\langle th\rangle$ for $t\in \mathbb Q_{>0}$ sufficiently small we may assume that $\int_X s_{\lambda}(E) h>0$.

Now set $\hat{X} = X\times \mathbb P^1$ and $\hat{E}= E\boxtimes \mathcal O_{\mathbb P^1}(1)$.    Observe $\hat{E}$ is nef on $\hat{X}$ and $|\lambda| = \dim(\hat{X})-2$.  So Theorem \ref{thm:derivedschurclassesareinHR} implies $$s_{\lambda}(\hat{E})\in \overline{\HR}(\hat{X}).$$ 

Let $\alpha\in H^{1,1}(X;\mathbb R)$ and denote by $\tau$ the hyperplane class on $\mathbb P^1$.   Also to ease notation
define 
$$\Omega: = s_{\lambda}(E)\in H^{d-1,d-1}(X;\mathbb R)\text{ and } \Omega':= s_{\lambda}^{(1)}(E)\in H^{d-2,d-2}(X;\mathbb R)$$   
so $s_{\lambda}(\hat{E}) = \Omega + \Omega'\tau$.   

Now define
$$ \hat{\alpha}  := \alpha - \kappa  \tau \text{ where }\kappa: = \frac{\int_X \alpha\Omega' h}{\int_X \Omega h}$$ 
so
$$\hat{\alpha} s_{\lambda}(\hat{E}) h = \hat{\alpha} (\Omega + \tau \Omega')h =0.$$
Also observe
$$\int_{\hat{X}} s_{\lambda}(\hat{E}) h^2 = \int_X \Omega' h^2>0$$
so the Hodge-Index inequality applied to $s_{\lambda}(\hat{E})$  yields
$$0\ge \int_{\hat{X}} \hat{\alpha}^2 s_{\lambda}(\hat{E}) =  \int_{\hat{X}} (\alpha^2 - 2\kappa \alpha\tau)(\Omega + \tau\Omega')
= \int_X \alpha^2 \Omega'  - 2\kappa\int_X \alpha\Omega.$$
Rearranging this gives \eqref{eq:schurhodgeimproved}.
\end{proof}

\subsection{Khovanskii-Tessier-type inequalities}

Let $X$ be smooth and projective of dimension $d$.    Suppose that $E,F$ are vector bundles on $X$,  and let $\lambda$ and $\mu$ be partitions of length $|\lambda|$ and $|\mu|$ respectively, and to avoid trivialities we assume $|\lambda| + |\mu| \ge d$.

\begin{definition}
We say a sequence $(a_i)_{i\in \mathbb Z}$ of non-negative real numbers is \emph{log concave} if
\begin{equation}\label{eq:logconcave}a_{i-1} a_{i+1}\le a_i^2 \text{ for all } i\end{equation}
\end{definition}
We note that for a finite sequence, say $a_i=0$ for $i<0$ and for $i>n$, log-concavity is equivalent to \eqref{eq:logconcave} holding in the range $i=1,\ldots,n-1$.

\begin{theorem}\label{thm:generalizedKT}
Assume $E,F$ are nef.  Then the sequence
\begin{equation}\label{eq:generalizedKT}i\mapsto \int_X  s_{\lambda}^{(|\lambda|+|\mu|-d-i)}(E)s_{\mu}^{(i)}(F) \end{equation}
is log-concave
\end{theorem}


Before giving the proof, some special cases are worth emphasising.

\begin{corollary}\label{cor:lambdaandmuard}
Suppose that $|\lambda| = |\mu|=d$.  Then the sequence
$$i\mapsto \int_X s_{\mu}^{(d-i)}(E) s_{\lambda}^{(i)}(F)$$
is log-concave
\end{corollary}

\begin{corollary}\label{cor:Enefandhnef}  Suppose that $|\lambda| =d$ and let   $h$ be a nef class on $X$.  Then the sequence
\begin{equation}i\mapsto \int_X s_{\lambda}^{(d-i)}(E) h^{d-i} \label{eq:logconcavewithh}
\end{equation}
is log-concave.    In particular the map
\begin{equation}i\mapsto \int_X c_i(E)h^{d-i} \label{eq:logconcavewithh:chern}
\end{equation}
is log-concave.
\end{corollary}
\begin{proof}[Proof of Corollary \ref{cor:Enefandhnef}]
By continuity we may assume that $h$ is ample.  Let $L$ be a line bundle with $c_1(L)=h$.  By rescaling $h$ we may, without loss of generality, assume $L$ is globally generated giving a surjection
$$\mathcal O^{\oplus f+1} \to L\to 0$$
for some integer $f$.   Let $F^*$ be the kernel of this surjection.  Then $F$ is a vector bundle of rank $f$ that is globally generated and hence nef.  Now set $\mu = (f)$, so $s_{\mu}^{(j)}(F) =  c_{f-j}(F)=h^{f-j}$.  
We now replace $i$ with $f-d+i$ in \eqref{eq:generalizedKT} (which is an affine linear transformation so does not affect log-concavity).  Note that
$$|\lambda| + |\mu| -d - (f-d+i) = |\lambda|-i,$$ so Theorem \ref{thm:generalizedKT} gives \eqref{eq:logconcavewithh}

Finally \eqref{eq:logconcavewithh:chern} follows upon letting $e:=\rank(E)$ and putting $\lambda = (e)$ so $s_{\lambda}^{(j)}(E) = c_{e-j}(E)$ so $s_{\lambda}^{(|\lambda|-i)}(E)= c_i(E)$.
\end{proof}

\begin{proof}[Proof of Theorem \ref{thm:generalizedKT}]
The first thing to note is that all the quantities in \eqref{eq:generalizedKT} are non-negative (see Remark \ref{rem:monomialsofderivedschurclasses}).  Also, we may as well assume $\rank(E)\ge \lambda_1$ and $\rank(F)\ge \mu_1$ else the statement is trivial.  

Set 
$$j = |\lambda| + |\mu| -d-i$$ and define
$$ a_i:=\int_X  s_{\lambda}^{(j)}(E)s_{\mu}^{(i)}(F)$$
so the task is to show that $(a_i)$ is log-concave.
We observe that $a_i=0$ if either $i$ or $j$ are negative, or $i> |\mu|$ or $j>|\lambda|$.  Thus the range of interest is
$$ \underline{i}: = \max\{ 0, |\mu|-d\} \le i \le \min \{ |\mu|, |\lambda| + |\mu|-d\} =: \overline{i}.$$
Fix such an $i$ in this range and consider
$$ \hat{X} = X\times \mathbb P^{j+1} \times \mathbb P^{i+1}.$$
Let $\tau_1$ be the pullback of the hyperplane class on $\mathbb P^{j+1}$ and $\tau_2$ the pullback of the hyperplane class on $\mathbb P^{i+1}$ and consider
$$\Omega = s_{\lambda} (E(\tau_1)) \cdot s_{\mu} (F(\tau_2)).$$
Observe that by construction $|\lambda| + |\mu| = d + i +j = \dim{\hat{X}} -2=: \hat{d}-2$.   Expanding $\Omega$ as a polynomial in $\tau_1,\tau_2$ one sees that the coefficient of $\tau_1^j\tau_2^i$ is precisely $s_{\lambda}^{(j)} s_{\mu}^{(i)}$.   Thus
$$\int_{\hat{X}} \Omega \tau_1 \tau_2 = \int_X s_{\lambda}^{(j)} s_{\mu}^{(i)} \int_{\mathbb P^{j+1}} \tau_1^{j+1} \int_{\mathbb P^{i+1}} \tau_2^{i+1} = \int_X s_{\lambda}^{(j)} s_{\mu}^{(i)} = a_i.$$
Similarly $\int_{\hat{X}} \Omega \tau_1^2 = a_{i-1}$ and $\int_{\hat{X}} \Omega \tau_2^2= a_{i+1}$.

Now, since $E(\tau_1)$ and $F(\tau_2)$ are nef on $\hat{X}$ we know from Theorem \ref{thm:monomialsderivedschurclassesareinHR} that $\Omega\in \overline{\HR}(\hat{X})$.   Thus the Hodge-Index inequality \eqref{eq:HI} applies with respect to the classes $\tau_1,\tau_2$ which is
\begin{equation}\int_{\hat{X}} \Omega \tau_1^2 \int_{\hat{X}} \Omega \tau_2^2 \le \left(\int_{\hat{X}} \Omega \tau_1 \tau_2 \right)^2\label{eq:EandFnefHI}\end{equation}
giving the log-concavity we wanted.
\end{proof}

\begin{remark}
In \cite{RossTomaHR} we gave a slightly different proof of \eqref{eq:logconcavewithh} which gave more, namely that if $X$ is smooth and $E$ and $h$ are ample then the map in question is strictly log-concave.  We expect that an analogous improvement can be made to Theorem \ref{eq:generalizedKT},  but it is not clear how this can be proved using the methods we have given here, since the bundle $F$ constructed in the above proof is only nef.  
\end{remark}

\begin{question}
Is there a natural statement along the lines of Theorem \ref{thm:generalizedKT} that applies to three or more nef vector bundles?  For instance perhaps it is possible to package characteristic numbers into a homogeneous polynomial that can be shown to be Lorentzian in the sense of Br\"{a}nd\'{e}n-Huh \cite{BrandenHuh}.
\end{question}

\begin{corollary}\label{cor:combinatoric}
Let $\lambda$ and $\mu$ be partitions, and let $d$ be an integer with $d\le |\lambda| + |\mu|$.    Assume  $x_1,\ldots,x_e,y_1,\ldots,y_f\in \mathbb R_{\ge 0}$.   Then the sequence
$$i\mapsto s^{(|\lambda| + |\mu| -d +i)}_{\lambda}(x_1,\ldots,x_e) s_{\mu}^{(i)}(y_1,\ldots,y_f)$$
is log concave.
\end{corollary}
\begin{proof}
By continuity we may assume the $x_i$ and $y_i$ are rational.  Furthermore, by clearing denominators, we may suppose they all lie in $\mathbb N$.    Then take $X=\mathbb P^d$ and $E = \bigoplus_{i=1}^e \mathcal O_{\mathbb P^d}(x_i)$ and $F  = \bigoplus_{i=1}^f \mathcal O_{\mathbb P^d}(y_i)$. Then for any symmetric polynomial $p$ of degree $\delta$ we have $p(E) = p(x_1,\ldots,x_e)\tau^{\delta}$ and similarly for $F$.  Thus what we want follows from Theorem \ref{thm:generalizedKT}.
\end{proof}

Putting $e=f$ we can consider
$$ u_i : = s^{(|\lambda| + |\mu| -d +i)}_{\lambda}s_{\mu}^{(i)}$$
as a polynomial in $x_1,\ldots,x_e$.   Still assuming $d\le |\lambda| + |\mu|$,  Corollary \ref{cor:combinatoric} says that
$$ (u_i^2 - u_{i+1} u_{i-1} )(x_1,\ldots,x_e) \ge 0 \text{ for any } x_1,\ldots,x_e\in \mathbb R_{\ge 0}.$$

\begin{question}
Is  $u_i^2 - u_{i+1} u_{i-1}$ monomial-positive (i.e.\ a sum of monomials with all non-negative coefficients)?
\end{question}

\begin{corollary}\label{cor:combinatoric2}
Let $\lambda$ be a partition and $x_1,\ldots,x_e\in \mathbb R_{\ge 0}$.  Then the sequence
$$i\mapsto s_{\lambda}^{(i)}(x_1,\ldots,x_e)$$
is log-concave.
\end{corollary}
\begin{proof}
By continuity we may assume $x_i\in \mathbb Q_{>0}$, and then by clearing denominators that they are all in $\mathbb N$.   Set  $d = |\lambda|$ and $X=\mathbb P^d$ and $E = \bigoplus_{j=1}^e \mathcal O_{\mathbb P^d}(x_i)$ and $h= c_1(E)$ which are both ample.    Then for any symmetric polynomial $p$ of degree $d$ in $e$ variables we have $\int_X p(E) = p(x_1,\ldots,x_e)$.  Thus Corollary \ref{cor:Enefandhnef} tells us that the map
$$i\mapsto s_{\lambda}^{(d-i)}(x_1,\ldots,x_e) (x_1 + \cdots x_e)^{d-i}=:a_i$$
is log-concave   That is $a_{i-1}a_{i+1}\le a_i^2$, and dividing both sides of this inequality by $(x_1+\ldots+x_e)^{2d-2i}$ gives that $i\mapsto s_{\lambda}^{(d-i)}(x_1,\ldots,x_e)$ is log-concave.  Replacing $d-i$ with $i$ does not change the log-concavity, so we are done.
\end{proof}

\begin{question}
Do Corollary \ref{cor:combinatoric} or Corollary \ref{cor:combinatoric2} have a purely combinatorial proof? 
\end{question}

\subsection{Lorentzian Property of Schur polynomials}

We end with a discussion on how our results relate to those of Huh-Matherne-M\'esz\'aros-Dizier \cite{Huhetal}.   To do so we need some definitions that come from \cite{BrandenHuh}.  A symmetric homogeneous polynomial $p(x_1,\ldots,x_e)$ of degree $d$ is said to be \emph{strictly Lorentzian} if all the coefficients of $p$ are positive and for any $\alpha\in \mathbb N^e$ with $\sum_j \alpha_j=d-2$ we have
$$\frac{\partial^{\alpha}p}{\partial x^{\alpha}}  \text{ has signature } (+,-,\ldots,-).$$
We say $p$ is \emph{Lorentzian} if it is the limit of strictly Lorentzian polynomials.

Any homogeneous polynomial $p$ of degree $d$ can be written as $p = \sum_{\mu} a_{\mu} x^{\mu}$  where the sum is over $\mu\in \mathbb Z_{\ge 0}^e$ with $\sum \mu_j= d$.   We write $[p]_{\mu}:= a_\mu$ for the coefficient of $x^{\mu}$.  The \emph{normalization} of $p$ is defined by
$$N(p) : = \sum_{\mu} \frac{a_{\mu}}{\mu!} x^{\mu}.$$

\begin{theorem}[Huh-Matherne-M\'esz\'aros-Dizier \protect{\cite[Theorem 3]{Huhetal}}]\label{thm:schurlorentzian}
The normalized Schur polynomials $ N(s_{\lambda})$ are Lorentzian.
\end{theorem}

Our proof needs a preparatory statement.   For this we set $$t_j(x_1,\ldots,x_e) = x_j \text{ for each } j=1,\ldots, e.$$ 
\begin{lemma}\label{lem:sillypolynomial}
Let $p(x_1,\ldots,x_e)$ be a 
homogeneous polynomial of degree $d$, let $e'$ be any integer satisfying $e'\ge\max_{1\le j\le e}\deg_{x_j}(p)$, where $\deg_{x_j}(p)$
is the degree of $p$ with respect to the indeterminate $x_j$, 
 and set
$$ q(x_1,\ldots,x_e) := x_1^{e'} \cdots x_e^{e'} p(x_1^{-1},\ldots, x_e^{-1}).$$
Let $\alpha\in \mathbb Z_{\ge 0}^e$ with $\sum_j \alpha_j=d-2$ and set $\beta_j: = e'-\alpha_j$.  Then
$$ \frac{\partial^{\alpha}}{\partial x^{\alpha}} N(p) =\frac{1}{2} \sum_{1\le i,j\le e} [ q t_i t_j]_{\beta} x_i x_j.$$
\end{lemma}
\begin{proof}
For $1\le i\le e$ set $\delta_{i}=(0,\ldots,0,1,0,\ldots,0)\in \mathbb Z^e$ with $1$ at the $i$-th position. Then if $p$ is written as $p = \sum_{\mu} a_{\mu} x^{\mu}$, we get 
$$ \frac{\partial^{\alpha}}{\partial x^{\alpha}} N(p) =\frac{1}{2}\sum_{1\le i,j\le e}a_{\alpha+\delta_{i}+\delta_{j}}x_{i}x_{j}= \frac{1}{2}\sum_{1\le i,j\le e} [ q t_i t_j]_{\beta} x_i x_j,$$
as one can check by  expanding $p$ in monomials.
\end{proof}

\begin{proof}[Proof of Theorem \ref{thm:schurlorentzian}] 
Take a   partition $\lambda=(\lambda_{1},\ldots,\lambda_{N})$ of $d:=|\lambda|$ with $0\le \lambda_N\le \cdots \le \lambda_1$ and assume $\lambda_1\le e$ else the statement is trivial.  Then   $d$  is the degree of $s_{\lambda}(x_1,\ldots,x_e)$. 
Note that by adding zero members to the partition $\lambda$ we may increase $N$ without changing the value of $s_{\lambda}$. We may therefore suppose that in our case $N\ge e$. 
The dual partition to $\lambda$ is defined by
$$\overline{\lambda}_i: = e-\lambda_{N-i} \text{ for } i=1,\ldots,N$$
so $|\overline{\lambda}| = Ne - |\lambda| = Ne-d$.  

Applying the definition
$$s_{\lambda} = \det \left(\begin{array}{ccccc} c_{\lambda_1} & c_{\lambda_1+1} &\cdots &c_{\lambda_1 +N-1}\\
c_{\lambda_2-1} & c_{\lambda_2} &\cdots &c_{\lambda_2 +N-2}\\
\vdots & \vdots & \vdots & \vdots\\
c_{\lambda_N-N+1} & c_{\lambda_N -N +2}&\cdots &c_{\lambda_N}\\
\end{array}\right)$$
to 
$$x_1^N \cdots x_e^N s_{\lambda}(x_1^{-1} ,\ldots,x_e^{-1})$$
and multiplying  each row of the matrix defining
$$ s_{\lambda}(x_1^{-1} ,\ldots,x_e^{-1})$$
with $x_1\cdots x_e$, we get
$$x_1^N \cdots x_e^N s_{\lambda}(x_1^{-1} ,\ldots,x_e^{-1})=$$ $$
\det \left(\begin{array}{ccccc} c_{e-\lambda_1} & c_{e-\lambda_1-1} &\cdots &c_{e-\lambda_1 -N+1}\\
c_{e-\lambda_2+1} & c_{e-\lambda_2} &\cdots &c_{e-\lambda_2 -N+2}\\
\vdots & \vdots & \vdots & \vdots\\
c_{e-\lambda_N+N-1} & c_{e-\lambda_N +N -2}&\cdots &c_{e-\lambda_N}\\
\end{array}\right)=
s_{\bar\lambda}(x_1,\ldots,x_e).$$

Thus
$$s_{\overline{\lambda}}(x_1,\ldots,x_e) = x_1^N \cdots x_e^N s_{\lambda}(x_1^{-1} ,\ldots,x_e^{-1})$$
and, equivalently,
$$s_{\lambda}(x_1,\ldots,x_e) = x_1^N \cdots x_e^N s_{\overline{\lambda}}(x_1^{-1} ,\ldots,x_e^{-1}).$$

It is tempting to now apply Lemma \ref{lem:sillypolynomial}, but before doing that we introduce a small perturbation.   For  $\epsilon>0$ set $\tilde{x}_j : = x_j + \epsilon \sum_p x_p$ and let
$$q_\epsilon(x_1,\ldots,x_e) := s_{\overline{\lambda}}(\tilde{x}_1,\ldots,\tilde{x}_e)$$
and
$$p_{\epsilon}(x_1,\ldots,x_e) := x_1^N \cdots x_e^N q_{\epsilon}(x_1^{-1},\cdots,x_e^{-1}),$$
so
\begin{equation}q_{\epsilon}(x_1,\ldots,x_e) = x_1^N \cdots x_e^N p_{\epsilon}(x_1^{-1},\cdots,x_e^{-1}).\label{eq:defqepsilon}\end{equation}
We will show that $N(p_{\epsilon})$ is strictly Lorentzian for small $\epsilon>0$,  which completes the proof since $p_\epsilon$ tends to $s_{\lambda}$ as $\epsilon$ tends to zero.

To this end, let $\alpha\in \mathbb Z_{\ge 0}^e$ with $\sum_j \alpha_j=d-2$ and set $\beta_j : =N- \alpha_j$ and 
$$ X := \prod_{j=1}^e \mathbb P^{\beta_j}. $$
Let $\tau_j$ denote the pulback of the hyperplane class on $\mathbb P^{\beta_j}$ to $X$, and set $h:= \sum_j \tau_j$ which is ample.   Next set
$$E :=  \bigoplus_{j=1}^e  \pi_j^* \mathcal O_{\mathbb P^{\beta_j}}(1) \text{ and } E' := E\langle \epsilon h\rangle.$$
Then $E$ is a nef vector bundle on $X$ and by construction $\dim X  = Ne-d+2 = |\overline{\lambda}|+2$.     So from Theorem \ref{thm:derivedschurclassesareinHR} we know $s_{\overline{\lambda}}(E) \in \overline{\HR}(X)$.  In fact by Remark \ref{rmk:comparisonwitholdpaper} we actually have $s_{\overline{\lambda}}(E')\in \HR(X)$ for sufficiently small $\epsilon>0$ and we assume henceforth this is the case.

Now by \eqref{eq:defqepsilon} and Lemma \ref{lem:sillypolynomial},
\begin{equation}\frac{\partial^{\alpha}}{\partial x^{\alpha}} N(p_{\epsilon}) =  \frac{1}{2}\sum_{1\le i,j\le e} [ q_{\epsilon} t_i t_j]_{\beta} x_i x_j\label{eq:proofschulorenzian1}\end{equation}
and our goal is to show that this has the desired signature.  But this is precisely what we already know, since thinking of $s_{\bar\lambda}(E') \tau_i \tau_j$ as a homogeneous polynomial in $\tau_1,\ldots,\tau_e$,  integrating over $X$ picks out precisely the coefficient of $\tau^{\beta}$, and as $E'$ has Chern roots $\tau_1 + \epsilon h, \cdots ,\tau_e+\epsilon h$ this becomes
$$\int_{X} s_{\bar\lambda}(E') \tau_i \tau_j = [ q_{\epsilon} t_i t_j]_{\beta}.$$
Hence the quadratic form in \eqref{eq:proofschulorenzian1} is precisely the intersection form 
$\frac{1}{2}Q_{s_{\bar\lambda}(E')}$ on $H^{1,1}(X)$, 
which has signature $(+,-,\ldots,-)$ and we are done.
\end{proof}

\begin{remark}
There is a lot of overlap between what we have here and the original proof in \cite{Huhetal}.  For instance we rely here on our Theorem that Schur classes of (certain) ample vector bundles have the Hodge-Riemann property, which in turn relies on the Bloch-Gieseker theorem and thus on the classical Hard-Lefschetz Theorem.  On the other hand,  \cite{Huhetal} relies on the fact that the volume function on a projective variety is Lorentzian, which is a facet of the Hodge-index inequalities (that are a consequence of the Hodge-Riemann bilinear relations).  

Also, instead of our cone classes discussed in \S\ref{sec:coneclasses}, the authors in \cite{Huhetal} use a different aspect of Schur classes that is also a degeneracy locus.  Finally we remark the use of the dual partition $\overline{\lambda}$ also appears crucially in \cite{Huhetal}.  Nevertheless there is a slightly different feel to the two proofs, and we leave it to the readers to decide if they consider them ``essentially the same" \cite{gowers_proofssame}.
\end{remark}

\bibliography{hrclassesbib.bib}

\begin{thebibliography}{10}

\bibitem{AissenSchoenbergWhitney}
Michael Aissen, I.~J. Schoenberg, and A.~M. Whitney.
\newblock On the generating functions of totally positive sequences. {I}.
\newblock {\em J. Analyse Math.}, 2:93--103, 1952.

\bibitem{BlochGieseker}
Spencer Bloch and David Gieseker.
\newblock The positivity of the {C}hern classes of an ample vector bundle.
\newblock {\em Invent. Math.}, 12:112--117, 1971.

\bibitem{BrandenHuh}
Petter Br\"{a}nd\'{e}n and June Huh.
\newblock Lorentzian polynomials.
\newblock {\em Ann. of Math. (2)}, 192(3):821--891, 2020.

\bibitem{Chen}
William Y.~C. Chen, Larry X.~W. Wang, and Arthur L.~B. Yang.
\newblock Schur positivity and the {$q$}-log-convexity of the {N}arayana
  polynomials.
\newblock {\em J. Algebraic Combin.}, 32(3):303--338, 2010.

\bibitem{Cvetkovski}
Zdravko Cvetkovski.
\newblock {\em Inequalities}.
\newblock Springer, Heidelberg, 2012.
\newblock Theorems, techniques and selected problems.

\bibitem{DemaillyPeternellSchneider}
Jean-Pierre Demailly, Thomas Peternell, and Michael Schneider.
\newblock Compact complex manifolds with numerically effective tangent bundles.
\newblock {\em J. Algebraic Geom.}, 3(2):295--345, 1994.

\bibitem{FultonIT}
William Fulton.
\newblock {\em Intersection theory}, volume~2 of {\em Ergebnisse der Mathematik
  und ihrer Grenzgebiete. 3. Folge. A Series of Modern Surveys in Mathematics
  [Results in Mathematics and Related Areas. 3rd Series. A Series of Modern
  Surveys in Mathematics]}.
\newblock Springer-Verlag, Berlin, second edition, 1998.

\bibitem{FultonLazarsfeld}
William Fulton and Robert Lazarsfeld.
\newblock Positive polynomials for ample vector bundles.
\newblock {\em Ann. of Math. (2)}, 118(1):35--60, 1983.

\bibitem{Gao}
Alice L.~L. Gao, Matthew H.~Y. Xie, and Arthur L.~B. Yang.
\newblock Schur positivity and log-concavity related to longest increasing
  subsequences.
\newblock {\em Discrete Math.}, 342(9):2570--2578, 2019.

\bibitem{gowers_proofssame}
Tim Gowers.
\newblock When are two proofs essentially the same?
\newblock
  \url{https://gowers.wordpress.com/2007/10/04/when-are-two-proofs-essentially-the-same}.
\newblock Accessed: 2021-05-12.

\bibitem{Gromov90}
M.~Gromov.
\newblock Convex sets and {K}\"{a}hler manifolds.
\newblock In {\em Advances in differential geometry and topology}, pages 1--38.
  World Sci. Publ., Teaneck, NJ, 1990.

\bibitem{Huhetal}
June Huh, Jacob Matherne, Karola M\`esz\`aros, and Avery St.~Dizier.
\newblock Logarithmic concavity of {S}chur and related polynomials, 2019.
\newblock arXiv:1906.09633.

\bibitem{Kempf}
G.~Kempf and D.~Laksov.
\newblock The determinantal formula of {S}chubert calculus.
\newblock {\em Acta Math.}, 132:153--162, 1974.

\bibitem{Lam}
Thomas Lam, Alexander Postnikov, and Pavlo Pylyavskyy.
\newblock Schur positivity and {S}chur log-concavity.
\newblock {\em Amer. J. Math.}, 129(6):1611--1622, 2007.

\bibitem{Lazbook2}
Robert Lazarsfeld.
\newblock {\em Positivity in algebraic geometry. {II}}, volume~49 of {\em
  Ergebnisse der Mathematik und ihrer Grenzgebiete. 3. Folge. A Series of
  Modern Surveys in Mathematics [Results in Mathematics and Related Areas. 3rd
  Series. A Series of Modern Surveys in Mathematics]}.
\newblock Springer-Verlag, Berlin, 2004.
\newblock Positivity for vector bundles, and multiplier ideals.

\bibitem{Miyoka}
Yoichi Miyaoka.
\newblock The {C}hern classes and {K}odaira dimension of a minimal variety.
\newblock In {\em Algebraic geometry, {S}endai, 1985}, volume~10 of {\em Adv.
  Stud. Pure Math.}, pages 449--476. North-Holland, Amsterdam, 1987.

\bibitem{Newton}
Isaac Newton.
\newblock {\em Arithmetica universalis: sive de compositione et resolutione
  arithmetica liber}.
\newblock 1707.

\bibitem{Okounkov-logconcave}
Andrei Okounkov.
\newblock Why would multiplicities be log-concave?
\newblock In {\em The orbit method in geometry and physics ({M}arseille,
  2000)}, volume 213 of {\em Progr. Math.}, pages 329--347. Birkh\"{a}user
  Boston, Boston, MA, 2003.

\bibitem{Richards}
Donald St.~P. Richards.
\newblock Log-convexity properties of {S}chur functions and generalized
  hypergeometric functions of matrix argument.
\newblock {\em Ramanujan J.}, 23(1-3):397--407, 2010.

\bibitem{RossTomaHR}
Julius Ross and Matei Toma.
\newblock {H}odge-{R}iemann bilinear relations for {S}chur classes of ample
  vector bundles, 2019.
\newblock arXiv:1905.13636, to appear in Ann. Sci. \'{E}cole Norm. Sup.

\bibitem{XiaoHighDegree}
Jian Xiao.
\newblock On the positivity of high-degree schur classes of an ample vector
  bundle, 2020.
\newblock arXiv:2007.12425.

\end{thebibliography}

\end{document}